\documentclass[12pt]{amsart}
\usepackage{amssymb}
\usepackage{amsthm}

\newtheorem{thm}{Theorem}[section]
\newtheorem{lem}[thm]{Lemma}
\newtheorem{prop}[thm]{Proposition}
\newtheorem{defn}[thm]{Definition}

\theoremstyle{remark}
\newtheorem{rem}[thm]{Remark}

\begin{document}

\newcommand{\FN}{\mathbb{N}}
\newcommand{\FZ}{\mathbb{Z}}  
\newcommand{\FC}{\mathbb{C}}  
\newcommand{\FR}{\mathbb{R}}  
\newcommand{\FQ}{\mathbb{Q}}  

\newcommand{\fa}{\mathfrak{a}}
\renewcommand{\a}{\mathfrak{a}}
\newcommand{\fb}{\mathfrak{b}}
\newcommand{\fc}{\mathfrak{c}}
\newcommand{\fm}{\mathfrak{m}}
\newcommand{\fo}{\mathfrak{o}}
\newcommand{\fp}{\mathfrak{p}}
\newcommand{\fq}{\mathfrak{q}}

\title[Special values of Hecke's L-funtion]{The behavior of Hecke's L-function of real quadratic fields at $s=0$}

\author{Byungheup Jun} 
\author{Jungyun Lee}
\thanks{The work of the first named author was supported by KRF-2007-341-C00006.}

\email{byungheup@gmail.com} \email{lee9311@kias.re.kr}
\address{School of Mathematics,  Korea  Institute for Advanced Study\\
Hoegiro 87, Dongdaemun-gu, Seoul 130-722, Korea}

\date{}
\maketitle

\begin{abstract}
For a family of real quadratic fields
$\{K_n=\FQ(\sqrt{f(n)})\}_{n\in \FN}$,  a  Dirichlet character
$\chi$ modulo $q$ and prescribed ideals $\{\fb_n\subset K_n\}$, we
investigate the linear behaviour of  the special value of partial
Hecke's L-function $L_{K_n}(s,\chi_n:=\chi\circ
N_{K_n},\fb_n)$ at $s=0$. We show that for $n=qk+r$,
$L_{K_n}(0,\chi_n,\fb_n)$ can be written as
$$\frac{1}{12q^2}(A_{\chi}(r)+kB_{\chi}(r)),$$
where $A_{\chi}(r),B_{\chi}(r)\in \FZ[\chi(1),\chi(2),\cdots,
\chi(q)]$ if a certain condition on $\fb_n$ in terms of its
continued fraction is satisfied. Furthermore, we write precisely
$A_{\chi}(r)$ and $B_{\chi}(r)$ using  values of the Bernoulli
polynomials. We describe how the linearity is used in solving
class number one problem for some families and recover the proofs
in some cases. Finally, we list some families of real quadratic
fields with the linearity.
\end{abstract}
\tableofcontents

\section{Introduction}
In this paper, we are mainly concerned with linear behaviour of the special values of Hecke's
$L$-function at $s=0$ for families of real quadratic fields.

Let $\{K_n=\FQ(\sqrt{f(n)})\}_{n\in \FN}$ be a family of real quadratic fields where $f(n)$ is a positive square free integer for each $n$.
For example $f(x)$ can be a polynomial with
 integer coefficients.

For a Dirichlet character $\chi$ modulo $q$, we have a ray class character $\chi_n:=
\chi\circ N_{K_n}$ for each $n$. 
Fixing an ideal $\fb_n$ in $K_n$ for each $n$, one obtains an indexed family of
partial Hecke L-functions $\{L_{K_n}(s,\chi_n,\fb_n)\}$, where the partial Hecke's L-function for 
$(K,\chi,\fb)$ is defined as 
$$
L_K(s,\chi,\fb) :=
\sum_{\begin{subarray}{l}\fa \sim \fb\\ \text{integral}\\ (q,\fa)=1 \end{subarray}}\chi(\fa)N(\fa)^{-s}.
$$
and $\fa\sim\fb$ means that $\fa=\alpha\fb$ for totally positive $\alpha\in K$.

Roughly speaking, if $L_{K_n}(0,\chi_n,\fb_n)$ can be written as linear polynomial in $k$ with coefficients
depending only on $r$ for $n=qk+r$, we say that $L_{K_n}(0,\chi_n,\fb_n)$ is linear.
\begin{defn}[Linearity]\label{linearity_definition}
When the special values of $L_{K_n}(s,\chi_n,\fb_n)$ at $s=0$ is expressed as
$$
L_{K_n}(0,\chi_n,\fb_n)=\frac{1}{12q^2}(A_{\chi}(r)+kB_{\chi}(r))
$$
for $n=qk+r$, $A_{\chi}(r), B_{\chi}(r)\in \FZ[\chi(1),\chi(2),\cdots\chi(q)]$,
we say that $L_{K_n}(0,\chi_n,\fb_n)$ is {\bf linear}.
\end{defn}

The ``linearity'' is originally observed by Bir\'o in his proof of Yokoi's conjecture(\cite{Biro1}). 

\begin{thm}[Yokoi's conjecture solved by Bir\'o]
If the class number of $\FQ(\sqrt{n^2+4})$ is $1$ then $n\leq 17$.
\end{thm}

In Yokoi's conjecture, we take $K_n = \FQ(\sqrt{n^2+4})$ and
$\fb_n=O_{K_n}$.
In page 88, 89 of \cite{Biro1},
Bir\'o expressed 
the special value of Hecke's $L$-function for $(K_n, \chi_n, O_{K_n})$ at $s=0$ for $n=qk+r$
\begin{equation}\label{equation}
L_{K_n}(0,\chi_n,\fb_n)=\frac{1}{q}(A_{\chi}(r)+kB_{\chi}(r)),
\end{equation}
where
\begin{equation*}\begin{split}
A_{\chi}(r)&=\sum_{0\leq C,D\leq q-1}\chi(D^2-C^2-rCD)\Big\lceil\frac{rC-D}{q}\Big\rceil(C-q),\\
B_{\chi}(r)&=\sum_{0\leq C,D\leq q-1}\chi(D^2-C^2-rCD)C(C-q).
\end{split}\end{equation*}

When $K_n$ is of class number $1$, the unique 
ideal class can be represented by any ideal $\fb_n$.
\textit{A priori} the partial Hecke $L$-function equals the total Hecke $L$-function up to multiplication by $2$(ie. 
$$L_{K_n}(0,\chi_n)= c L_{K_n}(0,\chi_n,O_{K_n})$$
for $c$ the number of narrow ideal classes).

From this identification, one can find  the residue of $n$ by
sufficiently many primes $p$ for which the class number of
$\FQ(\sqrt{n^2+4})$ is one. Moreover, from  the linearity, this
residue depends only on $r$. Consequently, one can tell whether
$p$ inerts or not in $\FQ(\sqrt{n^2+4})$. As we have a bound for a
smaller prime to inert depending on $n$, finally we have enough
conditions to list all $K_n$ of class number $1$.

Later in diverse works of Bir\'o, Byeon, Kim and the second named
author
(\cite{Biro2},\cite{Lee1},\cite{Lee2},\cite{Lee3},\cite{Lee4}),
other families $(K_n,\chi_n,\fb_n)$ that has linearity have been
discovered. Similarly,  developing  Biro's method, one can solve
the associated class number one problems.

In this paper, we give a criterion for $(K_n, \chi_n, \fb_n)$ to
have the linearity of the values $L_{K_n}(0,\chi_n,\fb_n)$ in
terms of the continued fraction expression of $\delta(n)$ where
$\fb_n^{-1} = [1,\delta(n)]:=\FZ+\delta(n)\FZ$. Let $[[a_0,a_1,\ldots,a_n]]$ be the purely periodic minus continued fraction
	$$[a_0,a_1,a_2,\ldots,a_n,a_0,a_1,\ldots],$$
	where
	$$
	[a_0,a_1,a_2,\ldots]:= a_0 + \cfrac{1}{a_1 +\cfrac{1}{a_2 + \cdots}}.
	$$
Our main theorem is as follows:

\begin{thm}[Linearity Criterion]\label{linearity_criterion}
Let $\{K_n=\FQ(\sqrt{f(n)})\}_{n\in \FN}$ be a family of real quadratic fields where $f(n)$ is a positive square free integer for each $n$.
Let $\chi$ be a Dirichlet character modulo $q$ for a positive integer $q$ and $\chi_n$ be a ray class character modulo $q$ defined by $\chi\circ N_{K_n}$.
Suppose $\fb_n$ is an integral ideal relatively prime to $q$ such
that $\fb_n^{-1} = [1,\delta(n)]$. Assume the continued
fraction expansion of  $\delta(n)-1$
$$\delta(n)-1=[[a_0(n),a_1(n),\cdots,a_{s-1}(n)]]$$
is purely periodic and of a fixed length $s$
independent of $n$ and  $a_i(n)=\alpha_i
n+\beta_i$ for some fixed $\alpha_i,\beta_i\in \FZ$.

If $N_{K_n}(\fb_n(C+D\delta(n)))$ modulo $q$  is a function only depending on $C$, $D$ and $r$ for $n=qk+r$, 
then $L_{K_n}(0,\chi_n,\fb_n)$ is linear.
\end{thm}

Furthermore, we give a precise  description of $A_{\chi}(r)$ and
$B_{\chi}(r)$ using  values of the Bernoulli polynomials
(Proposition \ref{AB}). From this description, for $n$ with
$h(K_n)=1$, as in Bir\'o's case, one can compute the residue of
$n$ modulo $p$ depending on the mod-$q$ residue $r$ of $n$. There
are possibly many $(q,p)$ pairs. The more pairs of $(q,p)$ we
have, the more we can restrict possible $n$. There are known many
families of which class number one problem can be solved in this
way. Many of known results can be recovered by ensuring the
linearity from continued fraction expansion and finding enough
$(q,p)$.

There are still other families of real quadratic fields with linearity whose class number one
problems are not yet answered. Morally, once we obtain reasonable class number one criterion, finding sufficiently many $(q,p)$-pairs should solve it.

This paper is composed as follows.
In Section 2, we describe the special value at $s=0$ of the partial Hecke L-function in terms of 
values of the Bernoulli polynomials.
Ssection 3 is devoted to the proof of our main theorem.
In Section 4, Bir\'o's method is sketched as
a prototype to apply the linearity. 
Finally in Section 5, we finish this paper with a possible generalization of the linearity criterion 
to polynomial of higher order.
%


\subsection*{Acknowledgment} 
We would like to thank Prof. Dongho Byeon for helpful comments and
discussions. We also thank the anonymous referee for careful reading and
many invaluable suggestions. The first named author wishes to thank Prof. Bumsig Kim and Prof. Soon-Yi Kang for warm supports and encouragements. 

\section*{Notations and conventions}
Throughout this article, we keep the following general notations and conventions. 
If we find it necessary,  we rewrite the notations in concrete terms at the place where it is used.

\begin{enumerate}
	\item $K$ is a real quadratic field. 
	\item For a real quadratic  field $K$, we fix an embedding $\iota:K \to \FR$. If there is no danger of confusion, we denote $\iota(\alpha)$ by an element $\alpha\in K$. $\alpha'$ denotes the conjugate of $\alpha$ as well as $\iota(\alpha')$. 
	\item For $\alpha\in K$, $N_{K}(\alpha)$ denotes the norm of $\alpha$ over $\FQ$. If there is no danger of confusion, we simply write $N(\alpha)$ to denote $N_K(\alpha)$. 
	For an integral ideal $\frak{a}$ of $K$, $N(\frak{a})$ denotes the norm of $\frak{a}$ defined to be $[\frak{o}_K, \frak{a}]$.
	\item For two linearly independent elements $\alpha, \beta \in K$ as a vector space over $\FQ$,  $[\alpha, \beta]$ denotes the lattice (ie. free abelian group) generated by $\alpha$ and $\beta$. A fractional ideal 
	$\fa$ of $K$ seen as a lattice is denoted by $[\alpha,\beta]$ if $\{\alpha,\beta\}$ is a free basis of $\fa$.  
	\item For a subset $A$ of $K$, we denote $A^+$ the set of totally positive elements in $A$. 
	\item $\chi$ is a fixed Dirichlet character of modulus $q$. 
	\item For a real number $x$, 
	$$
	\left< x\right> := 
		\begin{cases}
			x-[x], & \text{for $x\not\in \FZ$} \\
			1, & \text{for $x\in \FZ$}
		\end{cases}
	$$
	Equivalently, $\left< -\right>$ is the unique composition $\FR \xrightarrow{mod~\FZ} \FR/\FZ \to \FR$ that is identity on $(0,1]$. 
	
	\item For a real $x$, $[x]_1:=x-\left<x\right>$.
	\item For an integer $m$,  $\langle m\rangle_q$ denotes the residue of $m$ in $[1,q]$ by $q$(ie. $m= qk +\langle m\rangle_q $ for  $k\in \FZ$,  $\langle m\rangle_q  \in [1,q]\cap \FZ$.).
	\item $[a_0, a_1, a_2, ....]$ for positive integers $a_i$ denotes the usual continued fraction:
	$$
	[a_0,a_1,a_2,\ldots]:= a_0 + \cfrac{1}{a_1 +\cfrac{1}{a_2 + \cdots}}
	$$
	$[a_0,a_1,\ldots,a_{i-1},\overline{a_i,a_{i+1},\ldots,a_{i+j}}]$ denotes the continued fraction with periodic part 
	$(a_i,a_{i+1},\ldots,a_{i+j})$.
	
	$[[a_0,a_1,\ldots,a_n]]$ is the purely periodic continued fraction $$[a_0,a_1,\ldots,a_n, a_0,a_1,\ldots].$$
	
	\item $(a_0, a_1, a_2, \ldots)$ denotes the minus continued fraction:
	$$
	(a_0,a_1,a_2,\ldots):= a_0 - \cfrac{1}{a_1-\cfrac{1}{a_2 -\cdots}}
	$$
	
	$((a_0,a_1,\ldots,a_n))$ is the purely periodic minus continued fraction:
	$$
	(a_0,a_1,a_2,\ldots,a_n,a_0,a_1,\ldots)$$
	\item For an integer $s$, $\mu(s)=1$(resp. $\frac{1}{2}$) if $s$ is odd(resp. even).
\end{enumerate}

\section{Partial Hecke $L$-function}
Throughout this section, $K$ denotes a real quadratic field and $\fb$ is a fixed integral ideal of $K$
relatively prime to $q$. 

A {\it ray class character} modulo $q$ is a homomorphism 
$$\chi:I_K(q)/P_K(q)\rightarrow \FC^*,$$ 
where $I_K(q)$ is a group of fractional ideals of $K$ which is relatively prime to $q$ and $P_K(q)$ is a subgroup of principal ideals 
$(\alpha)$ for totally positive  $\alpha\equiv1\pmod{q}.$

Throughout this section, $\fb$ is an integral ideal such that 
$\fb^{-1}=[1,\delta]$ for $\delta \in K$ satisfying $0<\delta'<1$ and $\delta>2$.


Define $$F:=\{(C,D)\in \FZ^2 |0\leq C,D\leq q-1, ((C+D\delta)\fb,q)=1\}.$$

Let $E^+$(resp. $E_q^+$) be the set of totally positive units (resp. the set of totally positive
units congruent to $1$ mod $q$) in $K$. 
Then $E^{+}$ acts on the set $F$ by the rule
 $$\epsilon\ast(C+D\delta)=C'+D'\delta $$
where
 $ \epsilon \cdot (C+D\delta) +q\fb^{-1}=C'+D'\delta+q\fb^{-1}\,\,\text{for}\,\, \epsilon\in E^{+}.$

\begin{lem}\label{epsilon-length}
$(C,D)$ in $F$ is fixed by the action of $\epsilon$ if and only if
$\epsilon$ is in $E^+_q$.
\end{lem}

\begin{proof}
$(C,D)$ is fixed by $\epsilon \in E^+$
if and only if $(C+D\delta)(\epsilon-1) \in q \fb^{-1}.$
Since $(\fb(C+D\delta),q)=1$, the condition $(C+D\delta)(\epsilon-1) \in q \fb^{-1}$ is equivalent to $\epsilon\equiv 1\pmod{q}.$
\end{proof}


\begin{lem}\label{rel}
Suppose $0\le C,D \le q-1$.
Then the following are equivalent:
\begin{enumerate}
\item $(C,D)$ is in $F$.
\item For every $\alpha \in\frac{C+D\delta}{q}+\fb^{-1}$, the ideal
$q\alpha\fb$ is relatively prime to $q$.
\item For a $\alpha \in \frac{C+D\delta}{q}+\fb^{-1}$, the ideal $q\alpha\fb$ is relatively prime to $q$.
\end{enumerate}
\end{lem}

\begin{proof}
Suppose that $(q,(C+D\delta)\fb)=1$.

We have $\frac{q\alpha}{C+D\delta}\in 1+\frac{q}{C+D\delta}\fb^{-1}$ for
$\alpha\in\frac{C+D\delta}{q}+\fb^{-1}$. Thus $(q,\fb(C+D\delta))=1$ implies that
$$\frac{q\alpha}{C+D\delta}\equiv1\pmod{q}.$$
Since $$q\fb\alpha=\fb(C+D\delta)\frac{q\alpha}{C+D\delta},$$
we have
$$(q\fb\alpha,q)=1.$$

If $(q,(C+D\delta)\fb)\not=1$, then $(q,q\fb\alpha)\not=1$
for $\alpha\in \frac{C+D\delta}{q}+\fb^{-1},$
since for $\alpha\in \frac{C+D\delta}{q}+\fb^{-1},$ we have
$$q\fb\alpha\subset(C+D\delta)\fb+qO_K.$$
\end{proof}


Let $F'=F/E^+$ be the orbit space by the action of $E^+$ on $F$.
Let $\tilde{F'}$ a fundamental set of $F'$. 
Let $\epsilon$
be the totally positive fundamental unit.
The order of the action of
$\epsilon$
is $\lambda:=[E^+:E_q^+]$ by Lemma \ref{epsilon-length}.
Then  we can decompose $F$ as follows:
\begin{equation}
F=\bigsqcup_{i=0}^{\lambda-1} \epsilon^i \tilde{F'}.
\end{equation}
According to this decomposition of $F$, we can decompose further the partial Hecke's $L$-function:
\begin{prop} \label{index} Let $q$ be a positive integer. 
For an ideal $\fb\subset K$ relatively prime to $q$ and a ray class character $\chi$ modulo $q$, we have
\begin{equation*}
\begin{split}
&L_K(s,\chi,\fb) =
\sum_{\begin{subarray}{l}\fa \sim \fb\\integral\\ (q,\fa)=1 \end{subarray}}\chi(\fa)N(\fa)^{-s}\\
&= \sum_{(C,D)\in \tilde{F'}}\chi((C+D\delta)\fb)
 \sum_{
 \alpha\in(\frac{C+D\delta}{q}+\fb^{-1})^+/E_q^+}
 N(q\fb\alpha)^{-s}.
\end{split}
\end{equation*}
\end{prop}

\begin{proof}
For $\alpha_1, \alpha_2 \in(q^{-1}\fb^{-1})^+$, $q\alpha_1\fb = q\alpha_2\fb$ if and only if $\alpha_1/\alpha_2 \in E^+$.

So we have
$$
\sum_{\substack{\fa \sim\fb\\ integral \\ (q,\fa)=1}} \frac{\chi(\a)}{N(\a)^s} =
\sum_{\substack{\fa \sim q \fb\\ integral \\ (q,\fa)=1}} \frac{\chi(\a)}{N(\a)^s}
= \sum_{\substack{\alpha \in (q^{-1}\fb^{-1})^+/E^+ \\ (q,q\alpha\fb)=1} }
    \frac{\chi(q\alpha\fb)}{N(q\alpha\fb)^s}
$$

 We also have for a totally positive fundamental unit $\epsilon>1$
\begin{equation*}
\begin{split}
\sum_{\begin{subarray}{l}\alpha\in (q^{-1}\fb^{-1})^+/E_q^+\\(q,q\fb\alpha)=1 \end{subarray}}
\frac{\chi(q\fb\alpha)}{N(q\fb\alpha)^s}&=
\sum_{\begin{subarray}{l}\alpha\in (q^{-1}\fb^{-1})^+/E^+ \\
(q,q\fb\alpha) =1 \end{subarray}} \sum_{i=0}^{\lambda-1}
\frac{\chi(q\fb\alpha\epsilon^i)}{N(q\fb\alpha\epsilon^i)^s} \\
&=\lambda \cdot \sum_{\begin{subarray}{l}\alpha\in (q^{-1}\fb^{-1})^+/E^+ \\
(q,q\fb\alpha)=1 \end{subarray}}\frac{\chi(q\fb\alpha)}{N(q\fb\alpha)^s}  .
\end{split}
\end{equation*}

And  from Lemma \ref{rel}, we have
\begin{equation*}
\begin{split}
\sum_{\begin{subarray}{l}\alpha\in (q^{-1}\fb^{-1})^+/E_q^+\\ (q,q\fb\alpha)=1 \end{subarray}}\frac{\chi(q\fb\alpha)}{N(q\fb\alpha)^s}&=
\sum_{(C,D)\in F}\sum_{\begin{subarray}{l}\alpha\in (\frac{C+D\delta}{q}+\fb^{-1})^+/E_q^+\\ (q,q\fb\alpha)=1 \end{subarray}}\frac{\chi(q\fb\alpha)}{N(q\fb\alpha)^s} \\
&=
\sum_{(C,D)\in F}\sum_{\alpha\in (\frac{C+D\delta}{q}+\fb^{-1})^+/E_q^+}
\frac{\chi(q\fb\alpha)}{N(q\fb\alpha)^s} .
\end{split}
\end{equation*}

By equation (2), the above is equal to
$$\sum_{(C,D)\in \tilde{F'}}\sum_{i=0}^{\lambda-1}\sum_{\alpha\in (\frac{(C+D\delta)\epsilon^i}{q}+\fb^{-1})^+/E_q^+}\frac{\chi(q\fb\alpha)}{N(q\fb\alpha)^{s}}.$$
Since
$$\sum_{\alpha\in (\frac{(C+D\delta)\epsilon^i}{q}+\fb^{-1})^+/E_q^+}
\frac{\chi(q\fb\alpha)}{N(q\fb\alpha)^{s}}=\sum_{\alpha\in (\frac{(C+D\delta)}{q}+\fb^{-1})^+/E_q^+}\frac{\chi(q\fb\alpha\epsilon^i)}{N(q\fb\alpha \epsilon^i)^{s}},$$
the above also equal to
$$\lambda\cdot \sum_{(C,D)\in \tilde{F'}}\sum_{\alpha\in (\frac{C+D\delta}{q}+\fb^{-1})^+/E_q^+}\frac{\chi(q\fb\alpha)}{N(q\fb\alpha)^{s}}.$$

Note that for $\alpha\in(\frac{C+D\delta}{q}+\fb^{-1})^+$,
$q\fb\alpha$ and $(C+D\delta)\fb$ are in the same ray class modulo $q$.
Thus $\chi(q\fb\alpha)=\chi((C+D\delta)\fb)$.
This completes the proof.
\end{proof}


\subsection{Shintani-Zagier cone decomposition}

We review briefly the decomposition of $(\FR^2)^+$ into cones due to Shintani and Zagier in \cite{Zagier}, \cite{Shintani}, \cite{Van}.
This depends on a real quadratic field $K$ and a fixed ideal $\fa$ inside. Here for the sake of computation,  we fix $\fa=\fb^{-1}$ where $\fb$ is set as in the beginning of this section.

$K$ is embedded into $\FR^2$ by $\iota=(\tau_1,\tau_2)$, where $\tau_1,\tau_2$ are two real embeddings of $K$.
Especially the totally positive elements of $K$ lands on $(\FR^2)^+$.
We are going to describe the fundamental domain of  $(\frac{C+D\delta}{q}+\fb^{-1})^+/E_q^+$ embedded into $(\FR^2)^+$. 

The multiplicative action of ${E_q}^+$ on $K^+$ induces an action on $(\FR^2)^+$ by coordinate-wise multiplication: 
$$\epsilon\circ(x,y)=(\tau_1(\epsilon)x,\tau_2(\epsilon)y).$$

A fundamental domain $\frak{D}_{\FR}$ of $(\FR^2)^+/E_q^+$ 
is given by
\begin{equation}\label{fun}
\frak{D}_{\FR} := 
\{x\iota(1)+y\iota(\epsilon^{-\lambda})|x>0,y\geq0\} \subset (\FR^2)^+
\end{equation}
where $E_q^+=\left<\epsilon^{\lambda}\right>$ for an integer $\lambda$ and $\epsilon>1$ is the unique totally positive fundamental unit. 

If we take the convex hull of  $\iota(\fb^{-1})\cap(\FR^2)^+$ in $(\FR^2)^+$, the vertices on the boundary are  $\{P_i\}_{i\in\FZ}$ for $P_i \in \iota(\fb^{-1})$ and determined by the inequalities that
$P_0=\iota(1), P_{-1}=\iota(\delta)$ and $x(P_i)<x(P_{i-1})$ where $x(P_k)$ denotes the first coordinate of $P_k$ for 
$k\in\FZ$. 
Since any two consecutive boundary points make a basis of $\iota(\fb^{-1})$, we find that
 $$\left(\begin{array}{cc}0 & 1 \\-1 & b_i\end{array}\right)\left(\begin{array}{c}P_{i-1} \\P_{i}\end{array}\right)=\left(\begin{array}{c}P_i \\P_{i+1}\end{array}\right),$$
for an integer $b_i$.  It is easy to see that  $b_i\geq2$ from the convexity.
Thus we obtain
\begin{equation}\label{recu}
x(P_{i-1})+x(P_{i+1})=b_ix(P_i).
\end{equation}

Put $\delta_i:=\frac{x(P_{i-1})}{x(P_i)}>1$. Note that $\delta_0 = \delta$.
$\delta_i$ satisfies a recursive relation: 
$$\delta_i=b_i-\frac{1}{\delta_{i+1}},\quad \text{for $i\in\FZ$}.$$

Therefore
$$
\delta_i=
b_i-\cfrac{1}{b_{i+1}-\cfrac{1}{b_{i+2}-\dotsb}}=(b_i,b_{i+1},b_{i+2}\cdots).
$$

Let $\epsilon>1$ be the totally positive  fundamental  unit. 
$\epsilon$ moves a boundary point to another boundaty point preserving the order. 
Thus we have
\begin{equation}\label{pi}
\epsilon\circ P_i=P_{i-m},
\end{equation}
for a positive integer $m$. 
Therefore 
we obtain the following proposition.

\begin{prop}\label{pr}

\begin{enumerate}
  \item $\delta_{i+m}=\delta_{i}$ for all $i\in\FZ$. 
  \item $\delta_i=((b_i,b_{i+1},\cdots,b_{i+m-1}))=b_i-\cfrac{1}{b_{i+1}-\dotsb\cfrac{1}{b_{i+m-1}-\cfrac{1}{b_i-\dotsb}}}.$
  \item $\iota(\epsilon^{-1})=P_m$
  \item $\epsilon^{-1} \circ P_i=P_{i+m}$ 
  \item $\iota(\epsilon^{-\gamma})=P_{\gamma m}$
\end{enumerate}

\end{prop}
\begin{proof}
(1) $\delta_{i+m}=\frac{x(P_{i+m-1})}{x(P_{i+m})}=\frac{\epsilon x(P_{i-1})}{\epsilon x(P_i)}=\delta_i$.

(2) This is an immediate consequence of 1.

(3) From Eq. (\ref{pi}),
$$P_m=\epsilon^{-1}\circ P_{0}.$$
Since $P_0=\iota(1)$ and $\epsilon^{-1}\circ \iota(1)=\iota(\epsilon^{-1})$. 

(4) This is immediate from (\ref{pi}).

(5) It is trivial from (3) and (4). 
\end{proof}




From (\ref{fun}) and (4) of Proposition\ref{pr}, $\frak{D}_\FR$  the fundamental domain $(\FR^2)^+/E_q^+$ 
is further decomposed into 
$(\lambda\cdot m)$-disjoint union of smaller 
cones:
$$
\frak{D}_\FR = 
\bigsqcup_{i=1}^{\lambda m}\{xP_{i-1}+yP_i\,\,|\,\,x>0,\,\,y\geq0\}.
$$

Obviously the fundamental set of the quotient
$(\iota(\frac{C+D\delta}{q}+\fb^{-1}) \bigcap {(\FR^2)}^+)/E_q^+$ inside $\frak{D}_\FR$, which
we denote by $\frak{D}$
is given by a disjoint union:
$$
\frak{D} := \bigsqcup_{i=1}^{\lambda m}
\Big{(}\iota(\frac{C+D\delta}{q}+\fb^{-1})\bigcap\{xP_{i-1}+yP_i\,\,|\,\,x>0,\,\,y\geq0\} \Big{)}.
$$

Since $\{P_{i-1},P_i\}$ is a $\FZ$-basis of $\iota(\fb^{-1})$,
there is a unique $(x_{C+D\delta}^i,y_{C+D\delta}^i)\in (0,1]\times[0,1)$
such that
$$
x_{C+D\delta}^iP_{i-1}+y_{C+D\delta}^iP_i\in\iota(\frac{C+D\delta}{q}+\fb^{-1}),
$$
for each $i,C,D\in\FZ.$
Thus
\begin{equation}\label{set}
\begin{split}
\iota\big(\frac{C+D\delta}{q}+\fb^{-1}\big)\bigcap\{xP_{i-1}+yP_i\,\,|\,\,x>0,\,\,y\geq0\} \\
=\{(x_{C+D\delta}^i+n_1)P_{i-1}+(y_{C+D\delta}^i+n_2)P_i\,\,|\,\, n_1, n_2 \in\FZ_{\geq0}\}.
\end{split}
\end{equation}

In \cite{Yamamoto}, Yamamoto found a recursive relation satisfied by $(x_{C+D\delta}^i,y_{C+D\delta}^i)$: 
\begin{equation}\label{Yam}
\begin{split}
x_{C+D\delta}^{i+1} & =\langle  b_ix_{C+D\delta}^i+y_{C+D\delta}^i\rangle,\\
y_{C+D\delta}^{i+1} & =1-x_{C+D\delta}^i,
\end{split}
\end{equation}
where $\langle\cdot\rangle$ is as defined at the end of the introduction. 
(ie.
$
\langle x \rangle=
 x-[x]$ (resp. $1$) for $x\not\in \FZ$ (resp. for $x\in \FZ$)).((2.1.3) of {\it loc.  sit.}).

Let  $A_i:=x(P_i)$ for all $i\in\FZ$. Then from Eq.(\ref{set}), we obtain the following: 
\begin{equation}\label{e1}
\begin{split}
& \sum_{\alpha \in  (\frac{C+D\delta}{q}+\fb^{-1})^+/E_q^+}
\frac{1}{N(\alpha)^s} \\
=&\sum_{i=1}^{\lambda m}\sum_{n_1,n_2\geq0}N((x_{C+D\delta}^i+n_1)A_{i-1}+(y_{C+D\delta}^i+n_2)A_{i})^{-s}\\
=&\sum_{i=1}^{\lambda m}\sum_{n_1,n_2\geq0}N((x_{C+D\delta}^i+n_1)\delta_i+(y_{C+D\delta}^i+n_2))^{-s} A_i^{-s}.
\end{split}
\end{equation}

In \cite{Shintani}, Shintani  evaluated $\sum_{n_1,n_2\geq0}N((x+n_1)\delta+(y+n_2))^{-s}$ at nonpositive integers. 
In particular, the value at $s=0$ is expressed by first and second Bernoulli polynomials as follows:
\begin{lem}[Shintani]\label{shintani}
\begin{equation*}
\begin{split}
&\sum_{n_1,n_2\geq0}N((x+n_1)\delta+(y+n_2))^{-s}\Big{|}_{s=0}\\
=&\frac{\delta+\delta'}{4}B_2(x)+B_1(x)B_1(y)+\frac{1}{4}(\frac{1}{\delta}+\frac{1}{\delta'})B_2(y).
\end{split}
\end{equation*}
\end{lem}

Using this,  we have 
\begin{equation}\label{eqcd}
\begin{split}
& \sum_{\alpha \in  (\frac{C+D\delta}{q}+\fb^{-1})^+/E_q^+}
\frac{1}{N(\alpha)^s}\Big{|}_{s=0} \\
=&\sum_{i=1}^{\lambda m} \frac{\delta_i+\delta_i'}{4}B_2(x_{C+D\delta}^i)+B_1(x_{C+D\delta}^i)B_1(y_{C+D\delta}^i)+\frac{1}{4}(\frac{1}{\delta_i}+\frac{1}{\delta_i'})B_2(y_{C+D\delta}^i)
\end{split}
\end{equation}

Moreover, Yamamoto in the proof of Theorem 4.1.1 of \cite{Yamamoto} simplified  the above:
\begin{lem}[Yamamoto]\label{yamamoto}
\begin{equation*}
\begin{split}
&\sum_{i=1}^{\lambda m}\frac{\delta_i+\delta_i'}{4}B_2(x_{C+D\delta}^i)+\frac{1}{4}(\frac{1}{\delta_i}+\frac{1}{\delta_i'})B_2(y_{C+D\delta}^i)\\
=&\sum_{i=1}^{\lambda m}\frac{b_i}{2}B_2(x_{C+D\delta}^i)
\end{split}
\end{equation*}
\end{lem}

Finally, we have

\begin{equation}\label{fi}
\begin{split}
& \sum_{\alpha \in  (\frac{C+D\delta}{q}+\fb^{-1})^+/E_q^+}
\frac{1}{N(\alpha)^s}\Big{|}_{s=0} \\
=&\sum_{i=1}^{\lambda m} B_1(x_{C+D\delta}^i)B_1(y_{C+D\delta}^i)+\frac{b_i}{2}B_2(x_{C+D\delta}^i)
\end{split}
\end{equation}

\begin{lem}\label{epsilon}
Let $\epsilon$ be the totally positive fundamental  unit of $K$ and $\lambda:=[E^+:E_q^+]$. Then we have
$$
x_{C+D\delta}^{mi+j}=x_{\epsilon^{i}\ast(C+D\delta)}^j \quad\text{and}\quad y_{C+D\delta}^{mi+j}=y_{\epsilon^{i}\ast(C+D\delta)}^j,
$$
for $j=0,1,2,\cdots,m-1$.
\end{lem}
\begin{proof}
From (4) of  Proposition \ref{pr}, we have 
$$A_{mi+j}=\epsilon^{-i} A_j,$$
for any integer $i$. 

Thus
\begin{equation*}
\begin{split}
x_{C+D\delta}^{mi+j}A_{mi+j-1}+y_{C+D\delta}^{mi+j}A_{mi+j}&=\\
x_{C+D\delta}^{mi+j}\epsilon^{-i}A_{j-1}+y_{C+D\delta}^{mi+j}\epsilon^{-i}A_{j}
&\in \frac{C+D\delta}{q}+\fb^{-1}.
\end{split}
\end{equation*}
Therefore,
$$x_{C+D\delta}^{mi+j}A_{j-1}+y_{C+D\delta}^{mi+j}A_j\in \frac{\epsilon^{i}\cdot(C+D\delta)}{q}+\fb^{-1}.$$
\end{proof}
From Lemma \ref{epsilon} and the periodicity of $b_i$, we have
\begin{lem}\label{sum}
\begin{equation*}
\begin{split}
&\sum_{\alpha\in(\frac{C+D\delta}{q}+\fb^{-1})^+/E_q^+}\frac{1}{N(\alpha)^s}\Big{|}_{s=0}\\
=&\sum_{i=1}^{m}\sum_{j=0}^{\lambda -1} B_1(x_{\epsilon^{j}\ast(C+D\delta)}^i)B_1(y_{\epsilon^{j}\ast(C+D\delta)}^i)+
\frac{b_i}{2}B_2(x_{\epsilon^{j}\ast(C+D\delta)}^i).
\end{split}
\end{equation*}
\end{lem}

Finally, we have
\begin{prop}\label{hecke0} For a ray class character $\chi$ modulo $q$ and  an ideal $\fb$ of $K$ such that 
$$\fb^{-1}=[1,\delta]$$  
for $\delta\in K$ with $\delta>2$ and  $0<\delta'<1$, we have 
 
\begin{equation*}\begin{split}
&L_K(0,\chi,\fb)\\
&=\sum_{1\leq C,D \leq q}\chi((C+D\delta)\fb)\sum_{i=1}^{m} B_1(x_{(C+D\delta)}^i)B_1(y_{C+D\delta}^i)+\frac{b_i}{2}B_2(x_{C+D\delta}^i)
\end{split}
\end{equation*}
\end{prop}
\begin{proof}
From Proposition 2.3, we obtain
\begin{equation*}
\begin{split}
&L_K(0,\chi,\fb) \\
&= \sum_{(C,D)\in \tilde{F'}}\chi((C+D\delta)\fb)
 \sum_{
 \alpha\in(\frac{C+D\delta}{q}+\fb^{-1})^+/E_q^+}
 N(q\fb\alpha)^{-s}|_{s=0}.
\end{split}
\end{equation*}
 Lemma 2.8 implies that the above is equal to 
$$ \sum_{(C,D)\in \tilde{F'}}\chi((C+D\delta)\fb)
 \sum_{j=0}^{\lambda-1}\sum_{i=1}^{m} B_1(x_{\epsilon^j\ast(C+D\delta)}^i)B_1(y_{\epsilon^j\ast(C+D\delta)}^i)+\frac{b_i}{2}B_2(x_{\epsilon^j\ast(C+D\delta)}^i).$$
Since $(C+D\delta)\epsilon\fb=(C+D\delta)\fb$, 
the above is expressed as follows
$$ \sum_{(C,D)\in \tilde{F'}}
 \sum_{j=0}^{\lambda-1}\chi((C+D\delta)\epsilon^j\fb)\sum_{i=1}^{m} B_1(x_{\epsilon^j\ast(C+D\delta)}^i)B_1(y_{\epsilon^j\ast(C+D\delta)}^i)+\frac{b_i}{2}B_2(x_{\epsilon^j\ast(C+D\delta)}^i).$$
From \begin{equation}
F=\bigsqcup_{i=0}^{\lambda-1} \epsilon^i \tilde{F'},
\end{equation} we find that 
the above equals to 
$$\sum_{(C,D)\in F}\chi((C+D\delta)\fb)\sum_{i=1}^{m} B_1(x_{(C+D\delta)}^i)B_1(y_{(C+D\delta)}^i)+\frac{b_i}{2}B_2(x_{(C+D\delta)}^i).$$
If $((C+D\delta)\fb,q)\not=1$ then $\chi((C+D\delta)\fb)=0.$ Thus we complete the proof.
\end{proof}


\begin{rem}
It is important to note that the summation running over $C,D\in [1,q]$ is actually supported on 
$F$. This is justified by the twist of the mod $q$ Dirichlet character. Obviously, $F$ depends on
$\delta$ in $K$, but the twisted sum has invariant form of $\delta$ and $K$. This is a subtle
point  in the proof of the main theorem as we deal with family of the Hecke's $L$-values with respect
to a family $(K_n, \chi_n, \fb)$.
\end{rem}

\section{Proof of the main theorem}

In this section, we compute the special values of Hecke's L-function for a family of real quadratic fields. The computation is made using the expression of the L-value in the previous section. 
After the computation, it will be apparent that the linearity property comes sufficiently 
from the shape of the continued fractions in the family. 
This will complete the proof of Theorem \ref{linearity_criterion}.

This  gives a criterion that will recover several approaches of class number problems 
for some families of real quadratic fields.

Consider a family of real quadratic fields $K_n=\FQ(\sqrt{d_n})$, 
where $d_n$ is a positive square free integer. 
For a fixed Dirichlet character $\chi$ 
of modulus $q$, we associate a ray class character $\chi_n:=\chi\circ
N_{K_n/\FQ}$ for each $n$.
Let us fix an ideal $\fb_n$ of $K_n$ for each $n$. Then
we have a family of the Hecke's L-functions associated to $(K_n, \chi_n, \fb_n)$:
$$
L_{K_n}(s,\chi_n,\fb_n)=\sum_{\fa}\frac{\chi_n(\fa)}{N(\fa)^s}$$
where  $\fa$ runs over  integral ideals $\fa$ in the ray class represented by $\fb_n$. 




\subsection{Plan of the proof}

Assume that 
$$\fb_n^{-1}=[1,\delta(n)]$$
with  $\delta(n) > 2, 0<\delta(n)'<1$.
As discussed in Prop.\ref{pr}, $\delta(n)$ has a purely periodic minus continued fraction expansion: 
\begin{equation}
\begin{split}
\delta(n)=&((b_0(n),b_1(n),\cdots,b_{m(n)-1}(n)))\\= & b_0(n)-\cfrac{1}{b_1(n)-\cdots\cfrac{1}{b_{m(n)-1}(n)-\cfrac{1}{b_0(n)-\cdots}}}
\end{split}
\end{equation}
with $b_k(n) \ge 2$.

We extend the definition of $b_i(n)$  for all $i\in\FZ$ by requiring that
$b_{i+m(n)}(n)=b_i(n)$ for $i\in\FZ$.
Let $\delta_k(n)=((b_k(n),b_{k+1}(n),\cdots,b_{k+m(n)-1}(n)))$ and 
we define $\{A_k(n)\}_{k\in\FZ}$ by
$$
A_{-1}(n) = \delta(n), A_0(n)=1,\ldots, A_{k+1}(n)=A_{k}(n)/\delta_{k+1}(n).
$$

Then for fixed $C,D$ and $n$,  there is a unique $(x_{C+D\delta(n)}^i,y_{C+D\delta(n)}^i)$ such that
\begin{equation}
0<x_{C+D\delta(n)}^i\leq1,\,\,0\leq y_{C+D\delta(n)}^i<1,
\end{equation}
\begin{equation}
x_{C+D\delta(n)}^iA_{i-1}(n)+y_{C+D\delta(n)}^iA_i(n)\in\frac{C+D\delta(n)}{q}+\fb_n^{-1},
\end{equation}
for each $i\in\FZ$, as described in the previous section. 
This 
$(x^i_{C+D\delta(n)},y^i_{C+D\delta(n)})$ satisfies Yamamoto's recursive relation (\ref{Yam}) as follows:
\begin{equation}\label{yama}
x_{C+D\delta(n)}^{i+1}=\langle b_i(n)x_{C+D\delta(n)}^i+y_{C+D\delta(n)}^i\rangle,\,\,y_{C+D\delta(n)}^{i+1}=1-x_{C+D\delta(n)}^i.
\end{equation}





Now we recall a standard conversion formula of a plus continued fraction expansion to  minus continued fraction expansion:
\begin{lem}
Let $\delta-1$ be a purely periodic continued fraction:
$$
[[a_0,a_1,\cdots,a_{s-1}]].
$$
Then the minus continued fraction expansion of $\delta$ is
$$
((b_0, b_1, \cdots, b_{m-1})),
$$
where 
$$
b_i := \begin{cases} a_{2j}+2, & \text{for $i=S_j$}\\
2, & \text{otherwise}
\end{cases}
$$
for 
$$
S_j = \begin{cases} 0, &\text{for $j=0$}\\
	S_{j-1}+a_{2j-1}, & \text{for $j\geq1$}
	\end{cases}
$$
and the period 
$$
m = \begin{cases} a_1 + a_3+ a_5 \cdots + a_{s-1}=S_{\frac{s}{2}}, &\text{for even $s$}\\
                 a_0+a_1+a_2 \cdots+a_{s-1}=S_s, &\text{for odd $s$}
                 \end{cases}
$$
\end{lem}
\begin{proof} (See page 177, 178 of \cite{Zagier}). Actually if $s$ is an odd integer, the period $m$ is 
$$\sum_{i=1}^{s}a_{2i-1}=a_1+a_3+\cdots+a_{2s-1}=S_s.$$
Since $a_i$ has period $s$, we find that
$$a_1+1_3+\cdots+a_{2s-1}=a_0+a_1+a_2\cdots+a_{s-1}=\sum_{i=0}^{s-1}a_i.$$

\end{proof}

For the family of 
$\delta(n)\in K$, we assumed 
that 
$$
\delta(n)-1 = [[a_0(n),a_1(n), a_2(n),\ldots, a_{s-1}(n)]],
$$
has the same period for every $n$. 

Then $\delta(n)$ has purely periodic minus continued fraction expansion
$$
\delta(n) = ((b_0(n), b_1(n),\cdots,b_{m(n)-1}(n)))
$$
with $b_i(n)$, $S_j(n)$ and $m(n)$ defined by the same manner as in the previous lemma.

%




One should be aware that $m(n)$ vary with $n$, while the period of positive continued fraction $s$ is fixed. 
%

From Proposition \ref{hecke0} and recursive relation (15) of $(x_{C+D\delta(n)}^i,y_{C+D\delta(n)}^i)$, we have
\begin{equation}\begin{split}
&L_{K_n}(0,\chi_n,\fb_n) = \\
	&\sum_{1\le C,D\le q} \chi_n((C+D\delta(n))\fb_n) 
	\sum_{i=1}^{m(n)}\big(B_1(x^i_{C+D\delta(n)})B_1(y^i_{C+D\delta(n)}) + \frac{b_i(n)}2 B_2(x^i_{C+D\delta(n)})\big).
\end{split}\end{equation}

To check the linear behavior, it suffices to show that
\begin{equation}\label{sp}
\sum_{i=1}^{m(n)}\big(B_1(x_{C+D\delta(n)}^{i})B_1(y^i_{C+D\delta(n)})+\frac{b_i(n)}2B_2(x_{C+D\delta(n)}^{i})\big)
\end{equation}
is linear in $k$ with the coefficients determined only by $r$.

Because  $b_i(n)=2$ if  $i \ne S_j(n)$  for some $j$, 
we can divide the above into two parts: 
\begin{equation}
\begin{split}
&\sum_{l=1}^{s\mu(s)}\big(-B_1(x^{S_l(n)}_{C+D\delta(n)})B_1(x^{S_l(n)-1}_{C+D\delta(n)})+\frac{a_{2l}(n)+2}{2}B_2(x^{S_l(n)}_{C+D\delta(n)})\big)\\
&+\sum_{l=0}^{s\mu(s)-1}\sum_{i=S_l(n)+1}^{\substack{S_{l+1}(n)-1}}F(x^i_{C+D\delta(n)},x^{i-1}_{C+D\delta(n)})
\end{split}
\end{equation}
where $\mu(s)=\frac{1}{2}$ or $1$ for $s$ even or odd, respectively, and $F(x,y):=-B_1(x)B_1(y)+B_2(x)$.


If $C,D$ are fixed and there is no danger of misunderstand, $x_i(n)$  will simply mean $x_{C+D\delta(n)}^{i}$. 

Below is the behavior of $x_i(n)$, when $n$ varies.
The proof will be given later.\\

\fbox{\begin{minipage}{\textwidth}
1. $\{x_i(n)\}_{S_j(n)\leq i\leq S_{j+1}(n)}$ is an arithmetic progression mod $\FZ$ with 
common difference $\left<x_{S_{j}(n)+1}(n)-x_{S_{j}(n)}(n)\right>$.\\
2.  $\{x_i(n)\}_{S_j(n)\leq i\leq S_{j+1}(n)}$ has period $q$.\\
3.  $x_{S_{j}(n)}(n)$, $x_{S_{j}(n)-1}(n)$ and $x_{S_{j}(n)+1}(n)$ are invariant as $k$ varies for $n=qk+r$. 
\end{minipage}}\\

In short, $\{ x_i(n)\}$ is a `piecewise arithmetic progression'. 

As we have constrained that $a_i(n)=\alpha_i n+ \beta_i$,
 $\langle a_i(n)\rangle_q$ is independent of $k$ for $n=qk+r$ but depends only on $i$ and $r$.  

Define $\gamma_i(r)$ as follows:
\begin{equation}
\gamma_{i}(r):=\langle a_i(n)\rangle_q
\end{equation}
Then actually $\gamma_i(r)$ is $\left<a_i(r)\right>_q$.
Since $\{F(x_i(n),x_{i-1}(n))\}_{S_j(n)+1\leq i\leq S_{j+1}(n)-1}$ has period $q$ from 2 above, 
we obtain
\begin{equation}
\begin{split}
\sum_{i=S_l(n)+1}^{\substack{S_{l+1}(n)-1}}&F(x_i(n),x_{i-1}(n))\\
=&\sum_{i=S_l(n)+1}^{\substack{S_{l}(n)+\gamma_{2l+1}(r)-1}}F(x_i(n),x_{i-1}(n))+\kappa_{2l+1}(n)\sum_{i=S_l(n)+1}^{\substack{S_{l}(n)+q}}F(x_i(n),x_{i-1}(n))
\end{split}\end{equation}
where 
$a_{i}(n) = \kappa_i(n)q + \gamma_{i}(r)$ for an
integer $\kappa_i(n)$. 
Written precisely, 
\begin{equation}\label{kappa1}
\kappa_i(n) =\frac{a_{i}(n)-\gamma_{i}(r)}q.
\end{equation}
Since  $$\alpha_i r+\beta_i =q\tau_i(r)+\gamma_i(r)$$ for some integer $\tau_i(r)$,
we can write for $n=qk+r$
\begin{equation}\label{kappa2}
\kappa_i(n)=k\alpha_i+\tau_{i}(r)
\end{equation}

Using 3, $x_{S_l(n)}(n)$ and $x_{S_l(n)+1}(n)$ are sufficiently determined by the residue $r$ of $n$ by $q$. 
\textit{A priori} the summations $\sum_{i=S_l(n)+1}^{\substack{S_{l}(n)+\gamma_{2l+1}(r)-1}}F(x_i(n),x_{i-1}(n))$ and $\sum_{i=S_l(n)+1}^{\substack{S_{l}(n)+q}}F(x_i(n),x_{i-1}(n))$ are completely determined by 
$x_{S_l(n)}(n)$ and $x_{S_l(n)+1}(n)$ and remain unchanged while $k$ varies. 


Thus we conclude first that \\

\fbox{\begin{minipage}{\textwidth}
I. For $n=qk+r$,
$\sum_{i=S_l(n)+1}^{\substack{S_{l+1}(n)-1}}F(x_i(n),x_{i-1}(n))$ is linear function of $k$. 
\end{minipage}}\\

Using (\ref{kappa1}) and (\ref{kappa2}), we have
\begin{equation}
\begin{split}
&-B_1(x_{S_l(n)}(n))B_1(x_{S_l(n)-1}(n))+\frac{a_{2l}(n)+2}{2}B_2(x_{S_l(n)}(n))\\&=-B_1(x_{S_l(n)}(n))B_1(x_{S_l(n)-1}(n))+\frac{\alpha_{2l}qk+\tau_{2l}(r)q+\gamma_{2l}(r)+2}{2}B_2(x_{S_l(n)}(n))
\end{split}
\end{equation} 

Again after 3 we conclude that\\

\fbox{\begin{minipage}{\textwidth}
II. For $n=qk+r$, $-B_1(x_{S_l(n)}(n))B_1(x_{S_l(n)-1}(n))+\frac{a_{2l}(n)+2}{2}B_2(x_{S_l(n)}(n))$ is linear function on $k$.
\end{minipage}}\\

Additionally, we have

\fbox{\begin{minipage}{\textwidth}
III. $s$ and $\mu(s)$ is independent of $n$. 
\end{minipage}}\\

Altogether I,II and III imply the linearity of $\sum_{i=1}^{m(n)}-B_1(x_{i}(n))B_1(x_{i-1}(n))+\frac{b_i(n)}2B_2(x_{i}(n))$ in $k$ and the coefficients are function in $r$ for fixed $C,D$.
%

In sequal, 
we will clarify the properties 1, 2, 3 of $\{x_i(n)\}$.  
Also we give precise description $\sum_{i=1}^{m(n)}-B_1(x_{i}(n))B_1(x_{i-1}(n))+\frac{b_i(n)}2B_2(x_{i}(n))$ that will finish the proof of Theorem \ref{linearity_criterion}.

\subsection{Periodicity and invariance}
In this section, we will prove the properties 1, 2, 3 of $\{x_i(n)\}$ in the previous section. 
\begin{prop}
For $j\geq0,$ $\{x_i(n)\}_{S_j(n)\leq i\leq S_{j+1}(n)}$ is an arithmetic progression mod $\FZ$ with common difference $\langle x_{S_j(n)+1}(n)-x_{S_j(n)}(n)\rangle$.
\end{prop}
\begin{proof}
Since $b_i(n)=2$ for $S_j(n)+1\leq i\leq S_{j+1}(n)-1$, we have that
$$x_{i+1}(n)=\langle2x_{i}(n)-x_{i-1}(n)\rangle.$$
It implies that for $S_j(n)+1\leq i\leq S_{j+1}(n)-1$,
$$\langle x_{i+1}(n)-x_{i}(n)\rangle=\langle\langle2x_{i}(n)-x_{i-1}(n)\rangle-x_i(n)\rangle=\langle x_{i}(n)-x_{i-1}(n)\rangle.$$
\end{proof}

\begin{lem}\label{lem3.3}For $i\geq -1,$ we have
\begin{enumerate}
\item $qx_i(n)\in \FZ$.
\item $0<x_i(n)\leq1.$
\end{enumerate}
\end{lem}
\begin{proof}
Since $A_0(n)=1$, $A_{-1}(n)=\delta(n),$ from (13),(14) and (15) we find that 
$$x_0(n)=\bigg\langle\frac{D}{q}\bigg\rangle, \,\, x_{-1}(n)=1-\frac{C}{q}.$$
We also note that $b_i(n)\in \FZ$ for any $i\geq0$. Thus (15) implies above lemma.
\end{proof}

\begin{prop}
For $j\geq 0$ and $a_{2j+1}(n)\geq q$, $\{x_i(n)\}_{S_j(n)\leq i\leq S_{j+1}(n)}$ has period $q$. Explicitly we have
$$x_{S_j(n)+q+i}(n)=x_{S_j(n)+i}(n)\,\,\text{for}\,\,
0\leq i \leq a_{2j+1}(n)-q.$$
\end{prop}
\begin{proof} Note that $\{x_i(n) \,\,\text{mod} \,1\}_{S_j(n)\leq i\leq S_{j+1}(n)}$  is an arithemetic progression. Thus we have
$$x_{S_j(n)+q+i}(n)=\langle x_{S_j(n)+i}(n)+q\langle x_{S_j(n)+i}(n)-x_{S_j(n)+i-1}(n)\rangle\rangle,$$
for $0\leq  i\leq a_{2j+1}(n)-q.$ From Lemma \ref{lem3.3}, we find that 
$$q\langle x_{S_j(n)+i}(n)-x_{S_j(n)+i-1}(n)\rangle\in\FZ.$$
Thus
$$\langle x_{S_j(n)+i}(n)+q\langle x_{S_j(n)+i}(n)-x_{S_j(n)+i-1}(n)\rangle\rangle=\langle x_{S_j(n)+i}(n)\rangle.$$
Since $0<x_{S_j(n)+i}(n)\leq1,$ we finally have that 
$$\langle x_{S_j(n)+i}(n)\rangle=x_{S_j(n)+i}(n).$$
\end{proof}

For $0\leq r\leq q-1,$ we define
$$\Gamma_j(r):=\begin{cases}0, & \text{for $ j=0$}\\ \Gamma_j(r)+\gamma_{2j-1}(r), &\text{for $j \geq 1$}\end{cases},$$

where $\gamma_i(r)$ is defined as in (19). 
For $i\geq 0$, we put
$$ 
c_i(r)= \begin{cases}
\gamma_{2j}(r) + 2,  & \text{for $i=\Gamma_j(r)$} \\
2, & \text{otherwise}
\end{cases}
$$


Consider a sequence $\{\nu_{CD}^i(r)\}_{i\geq -1}$ with the  initial value and the recursive relation as follows:
\begin{equation*}
\nu_{CD}^{-1}(r)=\frac{q-C}{q},\quad \nu_{CD}^0(r)= \langle \frac{D}{q}\rangle 
\end{equation*}
and
$$\nu_{CD}^{i+1}(r)=\langle c_i(r)\nu_{CD}^i(r)-\nu_{CD}^{i-1}(r)\rangle. $$

If $C,D$ are fixed and clear from the context, we omit the subscript and abbreviate $\nu_{CD}^i(r)$ to $\nu_i(r).$ 

\begin{prop}\label{xi} With the notations above, for $j\geq0$ and $n=qk+r$, we have
$$x_{S_j(n)+i}(n)=\nu_{\Gamma_j(r)+i}(r)\quad\text{for}\,\,0\leq i\leq \gamma_{2j+1}(r)$$
\end{prop}

\begin{proof}

We use induction on $j$.

When $j=0.$ $S_0(n)= \Gamma_0(r) = 0$.
We need to show 
$x_i(n) = \nu_i(r)$ for $i\in [0,\gamma_1(r)]$. As we have seen in the proof of lemma \ref{lem3.3},
$$
x_0(n) = \langle\frac{D}{q}\rangle = \nu_0(r) ,\quad  x_{-1}(n) = 1-\frac{C}q=\nu_{-1}(r).
$$
Since $a_0(n)-\gamma_0(r)\in q\FZ$, using (\ref{yama}) and the recursive relation of $\nu_i(r)$, one can
easily check that
$$
x_1(n) = \langle (a_0(n)+2) \langle \frac{D}q\rangle + \frac{C}q \rangle = \langle (\gamma_0(r)+2) \nu_0(r) - \nu_{-1}(r) \rangle =  \nu_1(r)
$$

For  $1\leq i\leq \gamma_1(r)-1$, 
$x_i(n)$ and $\nu_i(r)$ satisfy the same recursive relation
$$x_{i+1}(n)=\langle 2x_i(n)-x_{i-1}(n)\rangle, \quad
\nu_{i+1}(r)=\langle 2\nu_i(r)-\nu_{i-1}(r)\rangle.$$

Thus we have
$$x_i(n)=\nu_i(r)\,\,\text{for} \,\, 0\leq i\leq \gamma_1(r).$$

Now assume that the proposition holds true for $j<j_0.$
From Proposition 3.4, we find that if $a_{2j_0-1}(n)\geq q$ then
\begin{equation}\label{equa1}
x_{S_{j_0-1}(n)+q+i}(n)=x_{S_{j_0-1}(n)+i}(n)\,\,\text{for}\,\,
0\leq i \leq a_{2j_0-1}(n)-q.
\end{equation}
Since $a_{2j_0-1}(n)-\gamma_{2j_0-1}(r)\in q\FZ$, we obtain
\begin{equation*}
\begin{split}
&x_{S_{j_0}(n)-1}(n)=x_{S_{j_0-1}(n)+a_{2j_0-1}(n)-1}(n)=x_{S_{j_0-1}(n)+\gamma_{2j_0-1}(r)-1}(n)\\
&=\nu_{\Gamma{j_0-1}(r)+\gamma_{2j_0-1}(r)-1}(r)=\nu_{\Gamma_{j_0}(r)-1}(r).
\end{split}
\end{equation*}
and
\begin{equation*}\begin{split}
&x_{S_{j_0}(n)}(n)=x_{S_{j_0-1}(n)+a_{2j_0-1}(n)}(n)\\
&=x_{S_{j_0-1}(n)+\gamma_{2j_0-1}(r)}(n)
=\nu_{\Gamma{j_0-1}(r)+\gamma_{2j_0-1}(r)}(r)=\nu_{\Gamma_{j_0}(r)}(r).
\end{split}\end{equation*}
Moreover from (15), we find that

\begin{equation}
\begin{split}
&x_{S_{j_0}(n)+1}(n)=\langle(a_{2j_0}(n)+2)x_{S_{j_0}(n)}(n)-x_{S_{j_0}(n)-1}(n)\rangle\\
&=\langle(\gamma_{2j_0}(r)+2)\nu_{\Gamma{j_0}(r)}(r)-\nu_{\Gamma_{j_0}(r)-1}(r)\rangle=\nu_{\Gamma_{j_0}(r)+1}(r).
\end{split}
\end{equation}

Since for $S_{j_0}(n)+1\leq i\leq S_{j_0+1}(n)-1$
$$x_{i+1}(n)=\langle 2x_i(n)-x_{i-1}(n)\rangle$$ and
for $\Gamma_{j_0}(r)+1\leq i\leq \Gamma_{j_0}(r)+\gamma_{2j_0+1}(r)-1=\Gamma_{j_0+1}(r)-1,$
$$\nu_{i+1}(r)=\langle 2\nu_i(r)-\nu_{i-1}(r)\rangle,$$
we have
$$x_{S_{j_0}(n)+i}(n)=\nu_{\Gamma_{j_0}(r)+i}(r)\,\,\text{for}\,\,0\leq i\leq \gamma_{2j_0+1}(r).
$$

\end{proof}

\subsection{Summations}
In this section we express 
$\sum_{i=1}^{m(n)}-B_1(x_{C+D\delta(n)}^{i})B_1(x_{C+D\delta(n)}^{i-1})+\frac{b_i(n)}2B_2(x_{C+D\delta(n)}^{i})$
using $\{\nu_{CD}^i(r)\}$. 

\begin{lem}Let $d_{l}(r):=\langle \nu_{\Gamma_{l}(r)+1}(r)-\nu_{\Gamma_{l}(r)}(r)\rangle$ and $[x]_1:=x-\langle x\rangle$.  Then for $1\leq \gamma\leq q$ and $n$ such that $\gamma\leq a_{2l+1}(n)$ and $n=qk+r$, we have
\begin{equation*}
\sum_{i=S_{l}(n)+1}^{S_{l}(n)+\gamma}(x_i(n)-x_{i-1}(n))^2
=\gamma d_{l}(r)^2
+(1-2d_{l}(r))[\nu_{\Gamma_{l}(r)}(r)+d_{l}(r)\gamma]_1
\end{equation*}
\end{lem}
\begin{proof}
Since $0<x_i(n)\leq1$, we have 
$$-1<x_i(n)-x_{i-1}(n)<1.$$
Thus $$x_i(n)-x_{i-1}(n)=\langle x_i(n)-x_{i-1}(n)\rangle +\psi_i(n),$$
where $$\psi_i(n)=
\begin{cases}-1, \,\,x_i(n)\leq x_{i-1}(n) \\
0,\,\,\,\,\,x_i(n)> x_{i-1}(n).
\end{cases}$$

As $$\langle x_{i+1}(n)-x_i(n)\rangle =\langle\langle 2x_i(n)-x_{i-1}(n)\rangle -x_i(n)\rangle =\langle x_i(n)-x_{i-1}(n)\rangle $$ for $S_l(n)+1\leq i\leq S_{l+1}(n)-1$, we have
$$\langle x_i(n)-x_{i-1}(n)\rangle =\langle x_{S_l(n)+1}(n)-x_{S_l(n)}(n)\rangle =\langle \nu_{\Gamma_l(r)+1}(r)-\nu_{\Gamma_l(r)}(r)\rangle =d_l(r).$$

Hence we have
$$x_i(n)-x_{i-1}(n)=d_l(r)+\psi_i(n).$$

Thus we obtain 
$$\sum_{i=S_{l}(n)+1}^{S_{l}(n)+\gamma}(x_i(n)-x_{i-1}(n))^2=\gamma {d_l(r)}^2+(1-2d_l(r))\sum_{i=S_{l}(n)+1}^{S_{l}(n)+\gamma} \psi_i(n)^2.$$
Note that
$\sum_{i=S_{l}(n)+1}^{S_{l}(n)+\gamma} \psi_i(n)^2$ equals
the number of $i$'s satisfying $x_i(n)\leq x_{i-1}(n)$ for $ S_l(n)+1\leq i\leq S_{l}(n)+\gamma.$

Therefore
$$\sum_{i=S_{l}(n)+1}^{S_{l}(n)+\gamma} \psi_i(n)^2=[x_{S_{l}(n)}(n)+d_l(r) \gamma]_1=[\nu_{\Gamma_{l}(r)}(r)+d_l(r) \gamma]_1.$$
\end{proof}

For simplicity, we let $$F(x,y):=-B_1(x)B_1(y)+B_2(x)=(x-\frac{1}{2})(\frac{1}{2}-y)+x^2-x+\frac{1}{6}.$$

\begin{lem}If $l \geq 0$ and $a_{2l+1}(n)\geq q,$
\begin{equation*}
\sum_{i=S_{l}(n)+1}^{S_{l}(n)+q}F(x_i(n),x_{i-1}(n))
=\frac{1}{12}\bigg[6\bigg(qd_{l}(r)^2
+(1-2d_{l}(r))[\nu_{\Gamma_{l}(r)}(r)+d_{l}(r)q]_1\bigg)-q \bigg]
\end{equation*}
And if $1\leq\gamma\leq q-1$ and 
$a_{2l+1}(n)\geq \gamma$,
\begin{equation*}
\begin{split}
&\sum_{i=S_{l}(n)+1}^{S_{l}(n)+\gamma}F(x_i(n),x_{i-1}(n)) \\
&=\frac{1}{12}\bigg[6\bigg(\gamma d_{l}(r)^2
+(1-2d_{l}(r))[\nu_{\Gamma_{l}(r)}(r)+d_{l}(r)\gamma]_1+B_2(x_{S_{l}(n)+\gamma}(n))-B_2(x_{S_{l}(n)}(n))\bigg)-\gamma \bigg]
\end{split}
\end{equation*}
where $B_2(x)$ is the second Bernoulli polynomial.

\end{lem}
\begin{proof}
We note that
$$F(x,y)=\frac{1}{2}(x-y)^2-\frac{1}{12}+\frac{1}{2}(B_2(x)-B_2(y)).$$
Thus we have
\begin{equation*}
\begin{split}
&\sum_{i=S_{l}(n)+1}^{S_{l}(n)+\gamma}F(x_i(n),x_{i-1}(n))\\
&=
\sum_{i=S_{l}(n)+1}^{S_{l}(n)+\gamma}\Big{[}\frac{1}{2}(x_{i}(n)-x_{i-1}(n))^2-\frac{1}{12}+\frac{1}{2}(B_2(x_i(n))-B_2(x_{i-1}(n)))\Big{]}.
\end{split}
\end{equation*}

We note that for $1\leq\gamma\leq q-1,$
$$\sum_{i=S_{l}(n)+1}^{S_{l}(n)+\gamma}B_2(x_i(n))-B_2(x_{i-1}(n))=B_2(x_{S_{l}(n)+\gamma}(n))-B_2(x_{S_{l}(n)}(n)).$$
and for $\gamma=q$ from the periodicity of $x_i(n)$ we have
$$\sum_{i=S_{l}(n)+1}^{S_{l}(n)+q}B_2(x_i(n))-B_2(x_{i-1}(n))=0.$$
\end{proof}



\begin{prop}\label{AB}
Suppose $\delta(n)-1=[[a_0(n),a_2(n),\cdots a_{s-1}(n)]]$, $a_i(n)=\alpha_i n+\beta_i$
for $\alpha_i ,\beta_i \in \FZ$ and  $a_i(r)=q\tau_i(r)+\gamma_i(r)$ for
$\gamma_i(r)=\langle a_i(r)\rangle_q$.
Let $d^{l}_{CD}(r):=\langle\nu^{\Gamma_{l}(r)+1}_{CD}(r)-\nu^{\Gamma_{l}(r)}_{CD}(r)\rangle$.
Then, for $n=qk+r$, we have
\begin{equation*}
\sum_{i=1}^{m(n)}-B_1(x^i_{C+D\delta(n)})B_1(y^i_{C+D\delta(n)})+\frac{b_i(n)}{2}B_2(x^i_{C+D\delta(n)})=\frac{1}{12}(A_{CD}(r)+kB_{CD}(r))
\end{equation*}
where
\begin{equation*}
\begin{split}
&A_{CD}(r):=\sum_{l=1}^{s\mu(s)} -12 B_1(\nu^{\Gamma_l(r)}_{CD}(r))B_1(\nu^{\Gamma_l(r)-1}_{CD}(r))+6(a_{2l}(r)+2)B_2(\nu^{\Gamma_l(r)}_{CD}(r))\\
&+\sum_{l=0}^{s\mu(s)-1} \Big{[}6 \Big{(}(\gamma_{2l+1}(r)-1){d^{l}_{CD}(r)}^2+(1-2d^{l}_{CD}(r))[\nu^{\Gamma_{l}(r)}_{CD}(r)+d^{l}_{CD}(r)(\gamma_{2l+1}(r)-1)]_1\\
&+B_2(\nu^{\Gamma_{l+1}(r)-1}_{CD}(r))-B_2(\nu^{\Gamma_{l}(r)}_{CD}(r))\Big{)}-\gamma_{2l+1}(r)+1 \\
&+\tau_{2l+1}(r) \Big{(}6(q{d^{l}_{CD}(r)}^2+(1-2d^{l}_{CD}(r))[\nu^{\Gamma_{l}(r)}_{CD}(r)+d^{l}_{CD}(r)q]_1)-q \Big{)}\Big{]}
\end{split}
\end{equation*}
and
\begin{equation}
\begin{split}
&B_{CD}(r):=\sum_{l=1}^{s\mu(s)}6q\alpha_{2l}B_2(\nu^{\Gamma_l(r)}_{CD}(r))\\
&+\sum_{l=0}^{s\mu(s)-1}\alpha_{2l+1}\Big{(}6(q{d^{l}_{CD}(r)}^2
+(1-2d^{l}_{CD}(r))[\nu^{\Gamma_{l}(r)}_{CD}+d^{l}_{CD}(r)q]_1)-q\Big{)}
\end{split}
\end{equation}
\end{prop}
\begin{proof}
From equation (18), we have
\begin{equation*}
\begin{split}
&\sum_{i=1}^{m(n)}B_1(x^i_{C+D\delta(n)})B_1(y^i_{C+D\delta(n)})+\frac{b_i(n)}{2}B_2(x^i_{C+D\delta(n)})\\
&=\sum_{l=1}^{s\mu(s)}[-B_1(x^{S_l(n)}_{C+D\delta(n)})B_1(x^{S_l(n)-1}_{C+D\delta(n)})+\frac{\alpha_{2l}qk+\tau_{2l}(r)q+\gamma_{2l}(r)+2}{2}B_2(x^{S_l(n)}_{C+D\delta(n)})]\\
&+\sum_{l=0}^{s\mu(s)-1}\sum_{i=S_l(n)+1}^{\substack{S_l(n)+q\alpha_{2l+1}k\\+q\tau_{2l+1}(r)+\gamma_{2l+1}(r)-1}}F(x^i_{C+D\delta(n)},x^{i-1}_{C+D\delta(n)})
\end{split}
\end{equation*}
From lemma 3.7, we have
\begin{equation*}
\begin{split}
&12\sum_{i=S_l(n)+1}^{\substack{S_l(n)+q\alpha_{2l+1}k\\+q\tau_{2l+1}(r)\gamma_{2l+1}(r)-1}}F(x^i_{C+D\delta(n)},x^{i-1}_{C+D\delta(n)})=12\sum_{i=S_l(n)+1}^{S_l(n)+\gamma_{2l+1}(r)-1}F(x^i_{C+D\delta(n)},x^{i-1}_{C+D\delta(n)})\\
&+12(\alpha_{2l+1}k+\tau_{2l+1}(r))\sum_{i=S_l(n)+1}^{S_l(n)+q}F(x^i_{C+D\delta(n)},x^{i-1}_{C+D\delta(n)})=\\
&6 \Big{(}(\gamma_{2l+1}(r)-1){d^{l}_{CD}(r)}^2+(1-2d^{l}_{CD}(r))[\nu^{\Gamma_{l}(r)}_{CD}(r)+d^{l}_{CD}(r)(\gamma_{2l+1}(r)-1)]_1 \\
&+B_2(x^{S_l(n)+\gamma_{2l+1}-1}_{C+D\delta(n)})-B_2(x^{S_{l}(n)}_{C+D\delta(n)})\Big{)}-(\gamma_{2l+1}(r)-1)\\
&+(\alpha_{2l+1}k+\tau_{2l+1}(r))\Big{(}6(q{d_{CD}^l(r)}^2+(1-2d_{CD}^l(r))[\nu_{CD}^{\Gamma_l(r)}(r)+d_{CD}^l(r)q]_1)-q\Big{)}
\end{split}
\end{equation*}
Since $x^{S_{l}(n)}_{C+D\delta(n)}=\nu^{\Gamma_l(r)}_{CD}(r)$, $x^{S_{l}(n)-1}_{C+D\delta(n)}=\nu^{\Gamma_l(r)-1}_{CD}(r)$ and $x^{S_{l}(n)+\gamma_{2l+1}(r)-1}_{C+D\delta(n)}=\nu^{\Gamma_{l+1}(r)-1}_{CD}$, we complete the proof.
\end{proof}

\subsection{End of the proof }

\begin{proof}

Since $\nu_{CD}^{\Gamma_l(r)}(r), \nu_{CD}^{\Gamma_l(r)-1}(r)$ and $d_{CD}^l(r)$ are in  $\frac{1}{q}\FZ$, we find that
$$q^2 A_{CD}(r), q^2B_{CD}(r)\in \FZ.$$
Moreover, we have
$$L_{K_n}(0,\chi_n,\fb_n)=\frac{1}{12q^2}\sum_{C,D}\chi_n(C+D\delta(n))(q^2A_{CD}(r)+kq^2B_{CD}(r)).$$
Since $\chi$ is a Dirichlet character of modulus $q$, if $n=qk+r$,
we can write
$$\chi_n(\fb_n(C+D\delta(n)))= F_{CD}(r),$$
for a function  $F_{CD}$. 
Note if $K_r$ is defined, $\chi_n(\fb_n(C+D\delta(n)))
=\chi_r(\fb_r(C+D\delta(r)))= F_{CD}(r)$. 
We warn the reader that the above expression does not make sense if $K_r$ and $\delta(r)$ 
are undefined.

If we set
$$A_{\chi}(r):=\sum_{C,D}F_{CD}(r) 
q^2A_{CD}(r)$$
and
$$B_{\chi}(r):=\sum_{C,D}F_{CD}(r) 
q^2B_{CD}(r),$$
we obtain  the proof.

\end{proof}

\section{Bir\'o's method}
Let $K_n$ be a family of real quadratic fields such that special value at $s=0$ of the Hecke $L$-function has linearity.
In \cite{Biro1} and  \cite{Biro2}, Bir\'o developed a way to find the residue of $n$ with $h(K_n)=1$ by certain
primes using the linearity.
In this section, we sketch Bir\'o's method.

Let $K_n=\FQ(\sqrt{d})$ for a square free integer  $d=f(n)$ and $D_n$ be the discriminant
$K_n$.
For an odd Dirichlet character $\chi:\FZ/q\FZ\rightarrow \FC^*$, 
$\chi_n$ denotes the ray class character defined as $\chi_n=\chi\circ N_{K_n}:I_n(q)/P_n(q)^+\rightarrow \FC^*$.  
$\chi_D=(\frac{D}{\cdot})$ denotes the Kronecker character.
Then  
the special value of the Hecke $L$-function at $s=0$ has a factorization
\begin{equation}
\begin{split}
L_{K_n}(0,\chi_n)=L(0,\chi)L(0,\chi\chi_{D_n})\\
=(\frac{1}{q}\sum_{a=1}^q a\chi(a))(\frac{1}{qD_n}\sum_{b=1}^{qD_n}
b\chi(b)\chi_{D_n}(b)),
\end{split}
\end{equation}

Let $\fb_n=O_{K_n}.$
Suppose that $L_{K_n}(0,\chi_n,\fb_n)$ is linear in the form: 
$$L_{K_n}(0,\chi_n,\fb_n)=\frac{1}{12q^2}(A_{\chi}(r)+kB_{\chi}(r)).$$
for $A_{\chi}(r),B_{\chi}(r)\in \FZ[\chi(1),\chi(2)\cdots\chi(q)]$.
Let $\epsilon_n$ be the fundamental unit of $K_n$.
Form Proposition 2.2 in \cite{Lee5}, we find that
$L_{K_n}(0,\chi_n,\fb_n)=L_{K_n}(0,\chi_n,(\epsilon_n)\fb_n).$
Thus if the class number of $K_n$ is one then
we have for $n=qk+r$
$$L_{K_n}(0,\chi_n)=\frac{c}{12q^2}(A_{\chi}(r)+kB_{\chi}(r))$$
where $c$ is the number of narrow ideal classes. 

Then we have
$$B_{\chi}(r)k+A_{\chi}(r)=\frac{12q}{c} \cdot(\sum_{a=1}^q a\chi(a))\cdot\Big{(}\frac{1}{qD_n}\sum_{b=1}^{qD_n}
b\chi(b)\chi_{D_n}(b)\Big{)}.$$

Let $L_{\chi}$ be the cyclotomic 
field 
generated by the values of $\chi$. 
Since $\frac{1}{qD_n}\sum_{b=1}^{qD_n}
b\chi(b)\chi_{D_n}(b)$ is integral 
in $L_{\chi}$, for a prime ideal $I$ of $L_{\chi}$ dividing $(\sum_{a=1}^q a\chi(a))$, we have
$$B_{\chi}(r)k+A_{\chi}(r) \equiv 0 
\pmod{I}.
$$
And if $I \nmid B_{\chi}(r)$ then
$$k\equiv-\frac{A_{\chi}(r)}{B_{\chi}(r)}
\pmod{I}.
$$
Since $n=qk+r$, we have
$$
n \equiv -q\frac{A_{\chi}(r)}{B_{\chi}(r)} +r \quad\pmod{I}.
$$
Moerover if $O_{L_{\chi}}/I=\FZ/p\FZ$,
the residue of $n$ modulo $p$ is expressed only in $A_\chi(r), B_\chi(r)$ and $r$ as above.

\bigskip
Below we arrange  all the necessary conditions of $q$ and $p$.\\

\fbox{\begin{minipage}{\textwidth}
\underline{Condition(*)}\\
1. $q$: odd integer \\
2. $p$: odd prime \\
3. $\chi$:  character with conductor $q$ \\
4. $I$: prime ideal in $L_{\chi}$  lying over $p$ \\
$\text{\,\,\,\,\,\,\,}I | (\sum_{a=1}^q a\chi(a))$ and  $O_{L_{\chi}}/I=\FZ/p\FZ$
\end{minipage}}\\

Note that the condition is independent of the family $\{K_n\}$, once the linearity
holds.

Let $S$ be the set of $(q,p)$  satisfying Condition(*).
$$
S = \bigcup_{q:\text{odd integer}} S_q
$$ where $S_q := \{(q,p)\in S\}$.




Finally we remark that for $(q,p)\in S$  we obtain the residue of $n=qk+r$ modulo $p$ for which the class number of $K_n$ is one.

The above method has been applied to find an upper bound of the discriminant of real quadratic fields with class number one in some families of Richaud-degert type where the linearity criterion is satisfied(cf. \cite{Biro2}, \cite{Biro1}, \cite{Lee1}, \cite{Lee3}). Together with properly developed class number one criteria for each cases, the class number problems could be solved.

It is easily checked that in fact the criterion is fulfilled by general families of Richaud-Degert type. Furthermore, there are still abundant examples such families of real quadratic fields satisfying the linearity criterion(cf. \cite{Mc}). For these, we have controlled behavior of the 
special values of Hecke's L-function at $s=0$ and Biro's method is directly applied for each cases. We expect many other meaningful problems for family of real quadratic fields than class number problem in arithmetic can be studied in this line.



\section{A generalization}

We conclude this section with a possible generalization of the
linearity of the special value of the Hecke $L$-function. This
generalization will be dealt in a separate paper \cite{Jun-Lee}.

As in the criterion for linearity, we set $K_n = \FQ(\sqrt{f(n)})$
and $\fb_n$ is an integral ideal of $K_n$. We assume
$$
\fb^{-1}_n = [1, \delta(n)],
$$
for
$$\delta(n)-1 = [a_1(n),a_2(n), \ldots, a_s(n)]$$
with $a_i(x) \in \FZ[x]$.

For a given conductor $q$, write $n=qk +r$ for $r=0,1,2,\ldots,
q-1$.  Suppose $N = \max_i \{\deg (a_i(x))\}$, then we obtain that
the special value of the partial $\zeta$-function of the ray class
of $\fb_n$ mod $q$ at $s=0$ is written as
$$
\zeta_{K_n,q}(0, (C+D\delta(n))\fb_n) = \frac{1}{12q^2}\big(A_0(r) + A_1(r) k +\cdots+A_N(r) k^N\big)
$$
for some rational integers $A_i$ depending only on $r$.

We have no application of this property in arithmetic. It will be
very interesting if one applies this in similar fashion as
Bir\'o's method as presented here.

\end{document}